\documentclass{amsart}
\usepackage[T1]{fontenc}
\usepackage[utf8]{inputenc}
\usepackage{lmodern}
\usepackage{microtype}
\usepackage{amsmath,amssymb,amsthm}
\usepackage[hidelinks]{hyperref}
\allowdisplaybreaks
\numberwithin{equation}{section}
\theoremstyle{plain}
\newtheorem{theorem}{Theorem}[section]
\newtheorem{proposition}[theorem]{Proposition}
\newtheorem{lemma}[theorem]{Lemma}
\newtheorem{corollary}[theorem]{Corollary}
\theoremstyle{remark}
\newtheorem{remark}[theorem]{Remark}
\DeclareMathOperator{\reg}{reg}
\DeclareMathOperator{\topdeg}{end}
\DeclareMathOperator{\NP}{NP}
\DeclareMathOperator{\conv}{conv}
\newcommand{\sat}{\mathrm{sat}}
\newcommand{\height}{\operatorname{height}}
\newcommand{\depth}{\operatorname{depth}}

\begin{document}

\title[Regularity under integral closure]
{Regularity under Integral Closure: A General Criterion and a Sharp Counterexample}

\author[Soumyadeep Misra]{Soumyadeep Misra}
\address{Department of Mathematics, University of Kansas, Lawrence, KS, USA}
\email{soumyadeep@ku.edu}

\subjclass[2020]{Primary 13B22; Secondary 13D02, 13F20, 14Q20}

\keywords{Castelnuovo--Mumford regularity, integral closure, saturation,
local cohomology, monomial ideals, Newton polyhedra}

\begin{abstract}
We study the behavior of Castelnuovo--Mumford regularity under integral closure. We prove a general cohomological criterion for the inequality $\reg(\overline J)\leq \reg(J)$, and in particular obtain it for every homogeneous ideal in a polynomial ring in at most three variables. The proof uses a high-degree comparison between $\overline J$  and its saturation. We also give a saturated monomial counterexample in four variables, showing that the three-variable result is sharp.

\end{abstract}
\maketitle

\section{Introduction}

Let $S=K[x_1,\ldots,x_n]$ be a standard graded polynomial ring over a field $K$, and let $J\subseteq S$ be a homogeneous ideal. We write $\overline J$ for the integral closure of $J$ and $\reg(J)$ for its Castelnuovo--Mumford regularity. The motivation for this paper comes from a conjecture of K\"uronya and Pintye concerning regularity and log-canonical thresholds. They conjectured that
\[
\reg(\overline{\mathcal I})\leq \reg(\mathcal I)
\]
for ideal sheaves on projective space; see \cite[Conjecture~12]{KuronyaPintye}. For homogeneous ideals in a polynomial ring, this naturally leads to the problem of comparing $\reg(J)$ and $\reg(\overline J)$. Positive results are known for several classes of monomial ideals, including ideals in small numbers of variables and monomial complete intersections; see \cite{CuiGongZhu,Javadekar}.

In this paper we establish a general cohomological criterion, valid in arbitrary dimension, under which integral closure does not increase regularity. As a consequence, we prove the inequality for every homogeneous ideal in a polynomial ring in at most three variables. We also show that this is sharp with respect to the number of variables by constructing a simple saturated monomial ideal in four variables for which integral closure increases regularity.

For a finitely generated graded $S$-module $M$, set
$a_i(M)=\topdeg H^i_{\mathfrak m}(M)$. Our main result is the following.

\begin{theorem}\label{thm:criterion-main}
Let $S=K[x_1,\ldots,x_n]$, let $0\neq J\subsetneq S$ be homogeneous, and
set
\[
        L=\overline J,\qquad I=L^{\sat},\qquad d=\dim(S/I).
\]
Suppose that $I\neq S$ and
\[
        \reg(S/I)=a_d(S/I)+d.
\]
Then $\reg(\overline J)\leq\reg(J)$.
\end{theorem}

The hypothesis says that the regularity of $S/I$ is detected by its top local cohomology. Cohen--Macaulay quotients satisfy this condition, giving the
following consequence.

\begin{corollary}\label{cor:cm-criterion-main}
With the notation of Theorem~\ref{thm:criterion-main}, suppose that $I\neq S$
and $S/I$ is Cohen--Macaulay. Then
\[
        \reg(\overline J)\leq\reg(J).
\]
\end{corollary}

The criterion applies automatically in three variables.

\begin{corollary}\label{thm:three-vars-main}
Let $S=K[x,y,z]$ and let $J\subseteq S$ be homogeneous. Then
\[
        \reg(\overline J)\leq \reg(J).
\]
\end{corollary}

The proof of Theorem~\ref{thm:criterion-main} has two parts. First, a local-cohomology argument gives control of $\reg((\overline J)^{\sat})$. The remaining issue is the difference between $\overline J$ and its saturation. For this, a valuative dehomogenization argument, valid in any number of variables, shows that
\[
\bigl((\overline J)^{\sat}\bigr)_t=(\overline J)_t
\qquad\text{for all }t\geq\delta(J),
\]
where $\delta(J)$ denotes the largest degree of a minimal homogeneous generator of $J$.
In particular, the saturation defect vanishes in degrees at least
$\delta(J)\leq\reg(J)$.
The saturation formula then gives the result. In three variables, after
removing a common factor and treating the $\mathfrak m$-primary case, one is
left with a height-two ideal. The saturated ideal
$(\overline J)^{\sat}$ has a one-dimensional Cohen--Macaulay quotient, so
Corollary~\ref{cor:cm-criterion-main} applies.

The three-variable result is sharp with respect to the number of variables.
Indeed, the monomial ideal
\[
I=(xz^2,\ yz^2,\ y^2z,\ y^2w)\subseteq K[x,y,z,w]
\]
satisfies
\[
\overline I=I+(xyzw),\qquad
\reg(I)=3,\qquad
\reg(\overline I)=4.
\]
Both $I$ and $\overline I$ are saturated. Thus some hypothesis in
Theorem~\ref{thm:criterion-main} is necessary in higher dimension, even when
there is no saturation defect on either side.

\section{Preliminaries}\label{sec:preliminaries}
Unless otherwise specified, $K$ is a field and $S=K[x_1,\ldots,x_n]$, with $\mathfrak m=(x_1,\ldots,x_n)$. For a graded $S$-module $M$, set
\[
\operatorname{end}(M)=\sup\{d\in\mathbb Z:M_d\neq0\},
\]
with $\operatorname{end}(0)=-\infty$. For a finitely generated graded $S$-module $M$, let $\reg(M)$ denote its Castelnuovo--Mumford regularity and set $a_i(M)=\operatorname{end}H^i_{\mathfrak m}(M)$, so that
\[
\reg(M)=\max_i\{a_i(M)+i\}.
\]
For a nonzero proper homogeneous ideal $I$, we use $\reg(I)=\reg(S/I)+1$. We write
\[
I^{\sat}=I:\mathfrak m^\infty=\bigcup_{q\geq1}(I:\mathfrak m^q)
\]
for the saturation of a homogeneous ideal $I$.

\begin{lemma}\label{lem:saturation-formula}
Let $I\subseteq S$ be a nonzero proper homogeneous ideal such that
$I^{\sat}\neq S$. Then
$
        \reg(I)=
        \max\left\{
        \reg(I^{\sat}),
        \topdeg(I^{\sat}/I)+1
        \right\}.
$
\end{lemma}

\begin{proof}
Set $T=I^{\sat}/I$. Since $T$ has finite length, $H^0_{\mathfrak m}(T)=T$ and
$H^i_{\mathfrak m}(T)=0$ for all $i\geq1$. From the exact sequence
$0\to T\to S/I\to S/I^{\sat}\to0$, we have
$T=H^0_{\mathfrak m}(S/I)$, while
$H^0_{\mathfrak m}(S/I^{\sat})=0$. Hence the long exact sequence in local
cohomology gives
$H^i_{\mathfrak m}(S/I)\cong H^i_{\mathfrak m}(S/I^{\sat})$ for all
$i\geq1$. Therefore
\[
\reg(S/I)=
\max\{\topdeg(T),\reg(S/I^{\sat})\}.
\]

Using $\reg(I)=\reg(S/I)+1$, the formula follows.

\end{proof}

\begin{lemma}\label{lem:newton-criterion}
Let $L\subseteq K[x_1,\ldots,x_n]$ be a monomial ideal, and let $G(L)$
denote its minimal monomial generating set. For $\alpha\in\mathbb N^n$,
one has
\[
 x^\alpha\in\overline L
 \quad\Longleftrightarrow\quad
 \alpha\in\conv\{\beta\in\mathbb N^n:x^\beta\in G(L)\}
              +\mathbb R_{\geq0}^n.
\]
\end{lemma}

\begin{proof}
This is the Newton polyhedron criterion for integral closures of
monomial ideals; see \cite[Chapter~1]{HunekeSwanson}.
\end{proof}

\begin{lemma}\label{lem:m-primary-reg}
Let $I\subseteq S$ be a homogeneous $\mathfrak{m}$-primary ideal. Then
$
    \reg(I)=\min\{t:\mathfrak{m}^t\subseteq I\}.
$
Consequently, if $I\subseteq J$ are homogeneous $\mathfrak{m}$-primary ideals, then
$\reg(J)\leq \reg(I)$.
\end{lemma}

\begin{proof}
Since $S/I$ is Artinian, its only nonzero local cohomology is
$H^0_{\mathfrak m}(S/I)=S/I$. Hence
$
        \reg(S/I)=\topdeg(S/I).
$
Thus $\reg(S/I)+1$ is the least integer $t$ for which
$\mathfrak m^t\subseteq I$. From $\reg(I)=\reg(S/I)+1$, the formula follows.
The final assertion is immediate from
$\mathfrak m^{\reg(I)}\subseteq I\subseteq J$.
\end{proof}

\section{General reductions and high-degree agreement}\label{sec:general-reductions}

Let $S=K[x_1,\ldots,x_n]$, and write
$\mathfrak m=(x_1,\ldots,x_n)$.
\begin{lemma}\label{lem:gcd-reduction}
Let $0\neq I\subseteq S$ be homogeneous. Let $h$ be the homogeneous gcd of
a minimal homogeneous generating set of $I$, and write $I=hJ$. Then 
\[
        \overline I=h\overline J,\qquad
        \reg(I)=\deg(h)+\reg(J),\qquad
        \reg(\overline I)=\deg(h)+\reg(\overline J).
\]
\end{lemma}

\begin{proof}
Since $S$ is a normal domain and $h$ is a nonzerodivisor,
\cite[Proposition~1.5.2]{HunekeSwanson} gives $
\overline{hJ}=h\overline J$. The two regularity equalities follow from the graded isomorphisms
$J(-\deg h)\cong hJ$ and
$\overline J(-\deg h)\cong h\overline J$.
\end{proof}

\begin{lemma}\label{lem:top-cohomology-monotonicity}
Let $J\subseteq I$ be nonzero proper homogeneous ideals of $S$, and set $d=\dim(S/I)$. Suppose that
$\dim(I/J)\leq d$ and
\[
        \reg(S/I)=a_d(S/I)+d.
\]
Then
\[
        \reg(J)\geq\reg(I).
\]
\end{lemma}

\begin{proof}
From
\[
0\longrightarrow I/J\longrightarrow S/J\longrightarrow S/I
\longrightarrow0
\]
and $\dim(I/J)\leq d$, Grothendieck vanishing gives
$H^{d+1}_{\mathfrak m}(I/J)=0$. Hence there is a surjection
\[
H^d_{\mathfrak m}(S/J)\twoheadrightarrow H^d_{\mathfrak m}(S/I),
\]
so $a_d(S/J)\geq a_d(S/I)$. Moreover,
$\reg(S/J)\geq a_d(S/J)+d$. Therefore
\[
\reg(S/J)\geq a_d(S/J)+d
\geq a_d(S/I)+d
=\reg(S/I).
\]
Using $\reg(J)=\reg(S/J)+1$ and
$\reg(I)=\reg(S/I)+1$ gives the result.
\end{proof}

\begin{corollary}\label{cor:cm-monotonicity}
Let $J\subseteq I$ be nonzero proper homogeneous ideals. Suppose that
$S/I$ is Cohen--Macaulay of dimension $d$ and $\dim(I/J)\leq d$. Then
\[
        \reg(J)\geq\reg(I).
\]
\end{corollary}

\begin{proof}
For a Cohen--Macaulay module $M$ of dimension $d$, regularity is
$a_d(M)+d$. Apply Lemma~\ref{lem:top-cohomology-monotonicity}.
\end{proof}

\begin{proposition}\label{prop:principal-mprimary}
Let $0\neq I\subseteq S$ be a homogeneous ideal.
Suppose that, after removing the homogeneous gcd of its minimal generators, the
remaining ideal is either $S$ or $\mathfrak{m}$-primary. Then
$
    \reg(\overline{I})\leq \reg(I).
$
\end{proposition}

\begin{proof}
Write $I=hJ$ as in Lemma~\ref{lem:gcd-reduction}. If $J=S$, then $I$
is principal, hence integrally closed. If $J$ is $\mathfrak m$-primary,
then $\overline J$ is also $\mathfrak m$-primary, so
Lemma~\ref{lem:m-primary-reg} gives
$
\reg(\overline{J})\leq \reg(J).
$
The result now follows from Lemma~\ref{lem:gcd-reduction}.
\end{proof}

\begin{lemma}\label{lem:valuative-criterion}
Let $R$ be a Noetherian domain with fraction field $K$, and let
$0\neq Q=(a_1,\ldots,a_s)\subseteq R$. For a rank-one discrete valuation
$v$ of $K$ that is nonnegative on $R$, set
$v(Q)=\min_i v(a_i)$. Then, for $f\in R$,
\[
f\in\overline Q
\quad\Longleftrightarrow\quad
v(f)\geq v(Q)
\]
for every such valuation $v$. Moreover, the forward implication holds for
every valuation of $K$ that is nonnegative on $R$, even if the valuation is
trivial or is not discrete.
\end{lemma}

\begin{proof}
By the valuative criterion for integral closure
\cite[Theorem~6.8.3]{HunekeSwanson}, $f\in\overline Q$ if and only if
$f\in QV$ for every rank-one discrete valuation ring $V$ with
$R\subseteq V\subseteq K$. If $v$ is the valuation corresponding to $V$,
then $R\subseteq V$ is equivalent to $v$ being nonnegative on $R$.
Moreover, since finitely generated ideals in a valuation ring are principal,
$QV=a_iV$ for some $i$ with $v(a_i)=\min_j v(a_j)=v(Q)$. Hence
$f\in QV$ if and only if $v(f)\geq v(Q)$, proving the claim.
The final assertion follows from
\cite[Proposition~6.8.10]{HunekeSwanson}.
\end{proof}

\begin{lemma}\label{lem:high-degree-agreement-general}
Let $J\subseteq S$ be a proper nonzero homogeneous ideal, and let
$g_1,\ldots,g_s$ be its minimal homogeneous generators. Set
$\delta(J)=\max_i\{\deg g_i\}$, $L=\overline J$, and $I=L^{\sat}$.
Then $I_d=L_d$ for all $d\geq\delta(J)$. Equivalently,
\[
\operatorname{end}(I/L)+1\leq\delta(J).
\]
Moreover, $\delta(J)\leq\reg(J)$.
\end{lemma}

\begin{proof}
Write $e_i=\deg g_i$, so $e_i\leq \delta(J)$ for every $i$. Let
$f\in I_d$ be homogeneous with $d\geq \delta(J)$. It suffices to show
that $f\in L$.

Fix a rank-one discrete valuation $v$ of $\operatorname{Frac}(S)$
that is nonnegative on $S$, and choose $j$ such that
$v(x_j)=\min_k v(x_k)$. Then
$v(x_i/x_j)=v(x_i)-v(x_j)\geq0$ for every $i$, so $v$ is nonnegative on
$A=(S_{x_j})_0=K[x_i/x_j:i\neq j]$.

For a homogeneous form $h$, write
$h^{(j)}=h/x_j^{\deg h}$, and set
\[
        J^{(j)}=(g_1^{(j)},\ldots,g_s^{(j)})\subseteq A.
\]
Set $B=S_{x_j}=A[x_j,x_j^{-1}]$. Since $f\in L^{\sat}$, there exists
$N\geq0$ such that $x_j^Nf\in L$. Hence $f^{(j)}\in LB$. Moreover,
$f^{(j)}$ has degree zero, so $f^{(j)}\in B_0=A$.

The ring $B$ is a Laurent polynomial extension of $A$, hence faithfully flat
over $A$, and $JB=J^{(j)}B$. Integral closure commutes with localization and
contracts under faithfully flat extensions
\cite[Propositions~1.1.4 and~1.6.2]{HunekeSwanson}. Therefore

\[
LB=\overline{JB}=\overline{J^{(j)}B},
\qquad
\overline{J^{(j)}B}\cap A=\overline{J^{(j)}}.
\]

Thus $f^{(j)}\in\overline{J^{(j)}}$.

The restriction of $v$ to $\operatorname{Frac}(A)$ may be trivial. However,
the forward implication in the final assertion of
Lemma~\ref{lem:valuative-criterion} applies and gives
\[
v(f^{(j)})\geq \min_i v(g_i^{(j)}).
\]
Equivalently,
\[
v(f)-d\,v(x_j)\geq
\min_i\{v(g_i)-e_i\,v(x_j)\}.
\]
Choose $i$ attaining the minimum. Then
\[
v(f)\geq v(g_i)+(d-e_i)v(x_j).
\]
Since $d\geq\delta(J)\geq e_i$ and $v(x_j)\geq0$, we get
$v(f)\geq v(g_i)\geq v(J)$. As this holds for every valuation in the
valuative criterion, $f\in\overline J=L$.

Thus $I_d\subseteq L_d$ for all $d\geq\delta(J)$. The reverse
inclusion follows from $L\subseteq L^{\sat}=I$, so $I_d=L_d$ for all
$d\geq\delta(J)$. Finally, the degrees of the minimal generators of $J$ are
bounded above by $\reg(J)$, and hence $\delta(J)\leq\reg(J)$.
\end{proof}

\begin{proof}[Proof of Theorem~\ref{thm:criterion-main}]
Set $r=\reg(J)$. Since integral closure preserves radicals,
$\dim(S/J)=\dim(S/L)$. Moreover, saturation does not change the localization
of $L$ at any prime different from $\mathfrak m$. Since $I=L^{\sat}\neq S$,
it follows that
\[
\dim(S/J)=\dim(S/L)=\dim(S/I)=d.
\]
Since $I/J\subseteq S/J$, we have $\dim(I/J)\leq d$.

 Hence
Lemma~\ref{lem:top-cohomology-monotonicity} gives
\[
        \reg(I)\leq\reg(J)=r.
\]
By Lemma~\ref{lem:high-degree-agreement-general},
$\topdeg(I/L)+1\leq\delta(J)\leq r$. Applying
Lemma~\ref{lem:saturation-formula} to $L$ yields
\[
        \reg(\overline J)=\reg(L)
        =\max\{\reg(I),\topdeg(I/L)+1\}\leq r.
\]
\end{proof}

\section{Reduction in three variables}\label{sec:reduction}

We now fix  $S=K[x,y,z]$, with
$\mathfrak m=(x,y,z)$.

\begin{proposition}\label{prop:reduction-height-two}
The inequality $\reg(\overline I)\leq\reg(I)$ holds for all homogeneous ideals $I\subseteq K[x,y,z]$ if and only if it holds for all homogeneous ideals $I\subseteq K[x,y,z]$ with $\operatorname{ht}(I)=2$.

\end{proposition}

\begin{proof}
The cases $I=0$ and $I=S$ are immediate. Assume
$0\neq I\subsetneq S$.
Write $I=hJ$, where $h$ is the homogeneous gcd of the minimal generators
of $I$. If $J=S$, then $I$ is principal; if $J$ is
$\mathfrak m$-primary, Proposition~\ref{prop:principal-mprimary} applies.
Assume that $J$ is proper and not $\mathfrak m$-primary. Its minimal
generators have no nonconstant common factor. If a height-one homogeneous
prime contained $J$, then, since $K[x,y,z]$ is a UFD, its generator would
divide every minimal generator of $J$. Hence $\height J\geq2$. Since
$J$ is not $\mathfrak m$-primary, $\height J\neq3$, and therefore
$\height J=2$. The inequality for $I$ now follows from that for
$J$ by Lemma~\ref{lem:gcd-reduction}.
\end{proof}
From now on,
$J\subseteq S$ denotes a homogeneous height-two ideal.

\begin{lemma}\label{lem:three-variable-cm}
Set $L=\overline J$ and $I=L^{\sat}$. Then $I\neq S$ and $S/I$ is
Cohen--Macaulay of dimension one.
\end{lemma}

\begin{proof}
Since $\sqrt L=\sqrt J$, the ideal $L$ has a height-two minimal prime
$\mathfrak p$. As $\mathfrak p\neq\mathfrak m$, localization at
$\mathfrak p$ is unchanged by saturation, so $\mathfrak p$ remains a
minimal prime of $I$. Thus $I\neq S$ and $\dim(S/I)=1$. Moreover,
$I$ is saturated, so $H^0_{\mathfrak m}(S/I)=0$. Hence
$\depth(S/I)\geq1=\dim(S/I)$, proving the assertion.
\end{proof}

\begin{proof}[Proof of Corollary~\ref{thm:three-vars-main}]
After Proposition~\ref{prop:reduction-height-two}, only the height-two
gcd-free case remains. Lemma~\ref{lem:three-variable-cm} and
Corollary~\ref{cor:cm-criterion-main} give the result.
\end{proof}

\section{The four-variable counterexample}\label{sec:counterexample}

Let $S=K[x,y,z,w]$ and
\[
        I=(xz^2,\ yz^2,\ y^2z,\ y^2w).
\]

\begin{theorem}\label{thm:four-var-main}
One has
\[
        \overline I=I+(xyzw),
        \qquad
        \reg(I)=3,
        \qquad
        \reg(\overline I)=4.
\]
In particular, $\reg(\overline I)>\reg(I)$.
\end{theorem}

\begin{proof}
First $xyzw\in\overline I$, since $(xyzw)^2=x^2y^2z^2w^2$
is divisible by $(xz^2)(y^2w)$, and hence $(xyzw)^2\in I^2$.

We claim that no other minimal monomial is added. By
Lemma~\ref{lem:newton-criterion}, if $x^ay^bz^cw^d\in\overline I$, then its
exponent vector lies in $\NP(I)$. The following inequalities are valid on
$\NP(I)$, since they are
valid on the exponent vectors
\[
        (1,0,2,0),\ (0,1,2,0),\ (0,2,1,0),\ (0,2,0,1)
\]
and have nonnegative coefficients:
\[
        a+b\geq 1,\qquad
        c+d\geq 1,\qquad
        b+c\geq 2,\qquad
        b+2c+2d\geq 4,\qquad
        2a+2b+c\geq 4.
\]
These inequalities force any monomial in $\overline I$ which is not already
divisible by one of $xz^2,yz^2,y^2z,y^2w$ to have
$a,b,c,d\geq 1$. Hence it is divisible by $xyzw$. Therefore
\[
        \overline I=I+(xyzw).
\]

It remains to compute regularities. We order the generators of $I$ as
\[
        g_1=xz^2,\qquad
        g_2=yz^2,\qquad
        g_3=y^2z,\qquad
        g_4=y^2w.
\]
Then
\[
        (g_1):g_2=(x),\qquad
        (g_1,g_2):g_3=(z),\qquad
        (g_1,g_2,g_3):g_4=(z).
\]
Since $I$ has linear quotients and is generated in degree $3$,
\cite[Lemma~4.1]{ConcaHerzog} gives $\reg(I)=3$.

For $\overline I$, add $g_5=xyzw$ to $g_1,\ldots,g_4$. The first three
colon ideals are as above, and $(g_1,g_2,g_3,g_4):g_5=(y,z)$, so
$\overline I$ also has linear quotients. Since $xyzw$ is not divisible by
any generator of $I$, it is minimal, and the generators of $\overline I$
have degrees $3$ and $4$. By \cite[Lemma~4.1]{ConcaHerzog},
$\reg(\overline I)=4$.
\end{proof}

\begin{remark}\label{rem:saturated-counterexample}
Both ideals in Theorem~\ref{thm:four-var-main} are saturated. Indeed, direct
monomial intersections give
\[
        I=(x,y)\cap (y^2,z^2)\cap (z,w).
\]
Using $\overline I=I+(xyzw)$, one likewise checks that
\[
\begin{split}
\overline I={}&(x,y)\cap(z,w)\cap(y,z^2)\cap(y^2,z)\\
&\cap(x,y^2,z^2)\cap(y^2,z^2,w).
\end{split}
\]
Each component is primary with radical different from $\mathfrak m$, hence saturated. Therefore both ideals are saturated, so the regularity jump is not caused by an $\mathfrak m$-primary component.
\end{remark}
\begin{remark}[Minimality with respect to the number of generators]
\label{rem:minimal-number-generators}
Although Theorem~\ref{thm:four-var-main} is sharp with respect to the
number of variables, its counterexample is not minimal with respect to the
number of generators. Indeed, let
\[
T=K[x,y,z,w,t]
\qquad\text{and}\qquad
J=(wt^{2},x^{2}t,x^{2}yz).
\]
The Hilbert--Burch resolution
\[
0\longrightarrow T(-5)^{2}
\longrightarrow T(-3)^{2}\oplus T(-4)
\longrightarrow J
\longrightarrow 0
\]
shows that $\reg(J)=4$. On the other hand, a direct application of
Lemma~\ref{lem:newton-criterion} gives
$\overline J=J+(xyzwt)$. Since $xyzwt$ is a minimal generator of $\overline J$, it follows that
$\reg(\overline J)\geq 5>4=\reg(J)$.
Thus three minimal generators already suffice.
\end{remark}
\section*{Acknowledgments}

The author thanks Hailong Dao for his encouragement and Sreehari Suresh Babu
for suggesting the problem and for many helpful discussions.


\begin{thebibliography}{9}

\bibitem{ConcaHerzog}
A.~Conca and J.~Herzog,
\newblock \emph{Castelnuovo--Mumford regularity of products of ideals},
\newblock Collect. Math. \textbf{54} (2003), 137--152.

\bibitem{CuiGongZhu}
Y.~Cui, C.~Gong, and G.~Zhu,
\newblock \emph{The regularity of monomial ideals and their integral closures},
\newblock preprint, arXiv:2509.15119.

\bibitem{HunekeSwanson}
C.~Huneke and I.~Swanson,
\newblock \emph{Integral Closure of Ideals, Rings, and Modules},
\newblock London Mathematical Society Lecture Note Series, vol.~336,
Cambridge University Press, Cambridge, 2006.

\bibitem{Javadekar}
O.~Javadekar,
\newblock \emph{A comparison of the regularity of certain classes of monomial ideals and their integral closure},
\newblock Arch. Math. \textbf{126} (2026), 351--363,
\newblock \href{https://doi.org/10.1007/s00013-026-02225-2}{doi:10.1007/s00013-026-02225-2}.

\bibitem{KuronyaPintye}
A.~K\"uronya and N.~Pintye,
\newblock \emph{Castelnuovo--Mumford regularity and log-canonical thresholds},
\newblock preprint, arXiv:1312.7778.

\end{thebibliography}
\end{document}